\documentclass[12pt,reqno]{amsart}

\usepackage{a4wide,amsthm}
\usepackage{comment}

\usepackage{amsmath}
\usepackage{amssymb}
\usepackage{amsfonts}
\usepackage{color}

\usepackage{graphicx}
\usepackage{subfig}
\usepackage{bm}
\usepackage{tikz}

\newtheorem{thm}{Theorem}

\newtheorem{question}[thm]{Question}

\newtheorem{lemma}{Lemma}[section]
\newtheorem{proposition}[lemma]{Proposition}
\newtheorem{theorem}[lemma]{Theorem}
\newtheorem{corollary}[lemma]{Corollary}

\theoremstyle{definition}

\theoremstyle{remark}
\newtheorem{remark}[lemma]{Remark}
\newtheorem{notation}[lemma]{Notation}

\newcommand{\dst}{\displaystyle}
\newcommand{\ffi}{\varphi}
\newcommand{\eps}{\varepsilon}

\def\R{{\mathbb{R}}}

\def\Z{{\mathbb{Z}}}

\def\dd{{\mathcal D}}

\def\ff{{\mathcal F}}

\newcommand{\norm}[1]{{\left\|{#1}\right\|}}
\newcommand{\ent}[1]{{\left[{#1}\right]}}
\newcommand{\abs}[1]{{\left|{#1}\right|}}

\newcommand{\ud}{\,\mathrm{d}}

\usepackage{hyperref}
\numberwithin{equation}{section}

\begin{document}

\title[The Littlewood problem]{The Littlewood problem and non-harmonic Fourier series}

\author[P. Jaming]{Philippe Jaming}
\address{Univ. Bordeaux, CNRS, Bordeaux INP, IMB, UMR 5251, F-33400 Talence, France}
\email{philippe.jaming@math.u-bordeaux.fr}

\author[K. Kellay]{Karim Kellay}
\address{Univ. Bordeaux, CNRS, Bordeaux INP, IMB, UMR 5251, F-33400 Talence, France}
\email{kkellay@math.u-bordeaux.fr}

\author[C. Saba]{Chadi Saba}
\address{Univ. Bordeaux, CNRS, Bordeaux INP, IMB, UMR 5251, F-33400 Talence, France}
\email{chadi.saba@math.u-bordeaux.fr}

\begin{abstract}
In this paper, we give a direct quantitative estimate of $L^1$
norms of non-harmonic trigonometric polynomials over large enough intervals. 
This extends the result by Konyagin \cite{Kon} and Mc Gehee, Pigno, Smith \cite{MPS} to the setting
of trigonometric polynomials with non-integer frequencies.

The result is a quantitative extension of a result by Nazarov \cite{Naz} and also covers a result
by Hudson and Leckband \cite{HL} when the length of the interval goes to infinity.
\end{abstract}



\date{}
\maketitle


\section{Introduction}

In 1948, in a paper written with Hardy \cite{1}, Littlewood investigated various extremal problems
for trigonometric sums, in particular for sums having only $0$ or $1$ as coefficients
$$
\sum_{k=1}^Ne^{2i\pi n_k t}
$$
where the $n_k$'s are distinct integers. In particular,
Littlewood speculated that the $L^1$-norm of such a sum might be minimal when the frequencies $(n_k)$
form an arithmetic sequence. Putting Littlewood's thoughts in a more precise form,
one is tempted to ask the following (still unanswered) question:

\begin{question}\label{Q1}
Is it true that, when $n_{-N}<n_{-1}<0<n_1<\cdots<n_N$ are integers
$$
\int_{-1/2}^{1/2}\left|\sum_{k=-N}^N e^{2i\pi n_k t}\right|\,\mathrm{d}t
\geq \int_{-1/2}^{1/2}\left|\sum_{k=-N}^N e^{2i\pi k t}\right|\,\mathrm{d}t.
$$
\end{question}

Littlewood did not explicitely ask this question
and made a safer guess. In view of the standard estimate
of the $L^1$-norm of the Dirichlet kernel (see e.g. \cite[Chapter 2, (12.1)]{Zy}
or \cite[Exercice 3.1.8]{Gr}).
$$
\int_{-1/2}^{1/2}\left|\sum_{k=-N}^N e^{2i\pi k t}\right|\,\mbox{d}t
\geq \frac{4}{\pi^2}\ln N,
$$
Littlewood conjectured that
$$
L_N:=\inf_{n_1<n_2<\cdots<n_N}\int_{-1/2}^{1/2}\left|\sum_{k=1}^N e^{2i\pi n_k t}\right|\,\mbox{d}t
\geq C\log N
$$
for some constant $C\leq\dfrac{4}{\pi^2}$. Note that in this formulation, the negative
frequencies are of no use.

There are strong reasons to believe that the answer to Question \ref{Q1} is positive. One of them
is that if the $n_k$'s are scattered instead of regularly spaced in the sense that $n_{k+1}>2n_k$
then it was well-known \cite[Chapter 5, (8.20)]{Zy} that the growth of the $L^1$-norm is much faster:
$$
\int_{-1/2}^{1/2}\left|\sum_{k=1}^N e^{2i\pi n_k t}\right|\,\mbox{d}t\geq C\sqrt{N}.
$$

The first non-trivial estimate was obtained by Cohen \cite{Coh} who proved that
$$
L_N\geq C(\ln N/\ln\ln N)^{1/8}
$$
for $N\geq 4$.
Subsequent improvements are due to Davenport \cite{Da}, Fournier \cite{Fo1}
and crucial contributions by Pichorides \cite{Pi1,Pi2,Pi3,Pi4}
leading to $L_N\geq C\ln N/(\ln\ln N)^2$.
Finally, Littlewood's conjecture was proved independently by Konyagin \cite{Kon} and Mc Gehee, Pigno, Smith \cite{MPS} in 1981. 
In both papers, Littlewood's conjecture is actually obtained as a corollary of a stronger result
(and they are not consequences of one another). In this paper, we are particularly interrested
in the result in \cite{MPS}:

\begin{thm}[Mc Gehee, Pigno \& Smith \cite{MPS}]
\label{th:MPS}
For $n_1<n_2<\cdots<n_N$ integers 
and $a_1,\ldots,a_N$ complex numbers,
$$
\int_{-1/2}^{1/2}\left|\sum_{k=1}^N a_ke^{2i\pi n_k t}\right|\,\mathrm{d}t\geq C_{MPS}\sum_{k=1}^N\frac{|a_k|}{k}
$$
where $C_{MPS}$ is a universal constant ($C_{MPS}=1/30$ would do).
\end{thm}

Note that, taking the $a_k$'s to have modulus $1$, one thus obtains a lower bound $L_N\geq C\ln N$.
The year after, Stegeman \cite{Steg} and Yabuta \cite{Yab} independently suggested
some modifications of the argument in \cite{MPS} that lead to a better bound of $L_N$, namely:

\begin{thm}[Stegeman \cite{Steg}, Yabuta \cite{Yab}]
\label{th:SY}
For every $N\geq 4$, $L_N\geq\dfrac{4}{\pi^3}\ln N$, that is,
if $n_1<n_2<\cdots<n_N$ are integers, then
$$
\int_{-1/2}^{1/2}\left|\sum_{k=1}^N e^{2i\pi n_k t}\right|\,\mathrm{d}t
\geq \frac{4}{\pi^3}\log N.
$$
\end{thm}

The actual constant is slightly larger but written like this, it is easier to compare
it to Question \ref{Q1}.
For a nice textbook proof of Theorem \ref{th:MPS} one may consult \cite{dvl},
\cite{CQ} (which also covers Nazarov's theorem below) or \cite{TB} (which also presents an
improvement of Theorem \ref{th:MPS}).
For the history of the question and related ones, we refer to \cite{BeKo,Fo2}.

The next question one might then ask is whether the previous results extend to non-integer frequencies.
This was first done by Hudson and Leckband \cite{HL} who used a clever perturbation argument
to prove the following:

\begin{thm}[Hudson \& Leckband \cite{HL}]
\label{th:HL}
For $\lambda_1<\lambda_2<\cdots<\lambda_N$ real numbers 
and $a_1,\ldots,a_N$ complex numbers,
$$
\lim_{T\to+\infty}\frac{1}{T}\int_{-T/2}^{T/2}\left|\sum_{k=1}^N a_ke^{2i\pi \lambda_k t}\right|\,\mathrm{d}t
\geq C_{MPS}\sum_{k=1}^N\frac{|a_k|}{k}
$$
where $C_{MPS}$ is the same constant as in Theorem \ref{th:MPS}.
\end{thm}

A further extension is due to Nazarov \cite{Naz} who showed that
such a result holds not only when $T\to+\infty$ but as soon as $T>1$:

\begin{thm}[Nazarov \cite{Naz}]
\label{th:Naz}
For $T>1$, there exists a constant $C_T$ such that, 
for $0<\lambda_1<\cdots<\lambda_N$ real numbers  such that $|\lambda_k-\lambda_\ell|\geq|k-\ell|$
and $a_1,\ldots,a_N$ complex numbers,
\begin{equation}
\label{eq:thnaz}
\int_{-T/2}^{T/2}\left|\sum_{k=1}^N a_ke^{2i\pi \lambda_k t}\right|\,\mathrm{d}t
\geq C_T\sum_{1\leq |k|\leq N}\frac{|a_k|}{|k|}.
\end{equation}
\end{thm}

An easy argument allows to show that
this theorem implies Theorem \ref{th:HL} but with worse constants. Further, the constants
in \cite{Naz} are not explicit but one may follow the computations
to show that $C_T\to0$ when $T\to 1$, though the speed
of convergence stays difficult to compute explicitley. In particular,
the question of the validity of this result when $T=1$ is open.

Our main aim here is to improve on Nazarov's proof to obtain a more precise and explicit 
estimate of this constant $C_T$. This also allows us to directly obtain the result of Hudson and Leckband. 
Moreover, we obtain the best constants known today:

\begin{theorem}
\label{th:AA}
Let $\lambda_1<\lambda_2<\cdots<\lambda_N$ be $N$ distinct real numbers
and $a_1,\ldots,a_N$ be complex numbers. Then
\begin{enumerate}
\renewcommand{\theenumi}{\roman{enumi}}
\item we have
$$
\lim_{T\to+\infty}\frac{1}{T}\int_{-T/2}^{T/2}\left|\sum_{k=1}^N a_ke^{2i\pi \lambda_k t}\right|\,\mathrm{d}t
\geq \frac{1}{26}\sum_{k=1}^N\frac{|a_k|}{k+1}.
$$

\item If further $a_1,\ldots,a_N$ all have modulus larger than $1$, $|a_k|\geq 1$ then
$$
\lim_{T\to+\infty}\frac{1}{T}\int_{-T/2}^{T/2}\left|\sum_{k=1}^N a_ke^{2i\pi \lambda_k t}\right|\,\mathrm{d}t
\geq \frac{4}{\pi^3}\ln N.
$$

\item Assume further that for $k,\ell=1,\ldots,N$, $|\lambda_k-\lambda_\ell|\geq |k-\ell|$, then
for $T\geq 72$ we have
$$
\frac{1}{T}\int_{-T/2}^{T/2}\left|\sum_{k=1}^N a_ke^{2i\pi \lambda_k t}\right|\,\mathrm{d}t
\geq \frac{1}{122}\sum_{k=1}^N\frac{|a_k|}{k+1}.
$$
\end{enumerate}
\end{theorem}

Let us make a few comments on the result. First, the limit in the statement of the result are well-known
to exist and are the Besikovitch norms of $\dst\sum_{k=1}^N a_ke^{2i\pi \lambda_k t}$.
Further, when the $\lambda_k$'s are all integers, then
$$
\lim_{T\to+\infty}\frac{1}{T}\int_{-T/2}^{T/2}\left|\sum_{k=1}^N a_ke^{2i\pi \lambda_k t}\right|\,\mathrm{d}t
=\int_{-1/2}^{1/2}\left|\sum_{k=1}^N a_ke^{2i\pi \lambda_k t}\right|\,\mathrm{d}t
$$
so that we recover Theorems \ref{th:MPS} and \ref{th:SY} (with the best constants today). 
As the constants in Theorem \ref{th:AA} are the best known for $C_{MPS}$ we also recover Theorem \ref{th:HL}
while we only recover Theorem \ref{th:Naz} for large enough $T$ which is due to the strategy of proof (see below).

Further, note that the left hand side in Theorem \ref{th:AA}{\it i)} and {\it ii)} is unchanged if one replaces the 
$\lambda_k$'s by $\alpha\lambda_k+\beta$. In the proof we will thus assume that
$|\lambda_k-\lambda_\ell|\geq |k-\ell|$ (or equivalently that $\lambda_{j+1}-\lambda_j\geq 1$). 
In Theorem \ref{th:AA}{\it iii)} this restriction only affects
the critical $T$ for which our proof works.

Finally, to better understand the difference between assertions {\it i)} and {\it iii)}, one may consider $L^2$-norms 
instead of $L^1$ norms. Recall that
\begin{equation}
\label{eq:ortho}
\lim_{T\to+\infty}\frac{1}{T}\int_{-T/2}^{T/2}e^{2i\pi\lambda t}e^{-2i\pi\mu t}\,\mbox{d}t=\delta_{\lambda,\mu}
\end{equation}
so that the family $\{e^{2i\pi \lambda_j t}\}$ is orthonormal and
$$
\lim_{T\to+\infty}\frac{1}{T}\int_{-T/2}^{T/2}\left|\sum_{k=1}^N a_ke^{2i\pi \lambda_k t}\right|^2\,\mathrm{d}t
=\sum_{k=1}^N|a_k|^2.
$$
On the other hand, the celebrated Ingham inequality (\cite{In}, see also \cite{BKL}) shows that
$$
\frac{1}{T}\int_{-T/2}^{T/2}\left|\sum_{k=1}^N a_ke^{2i\pi \lambda_k t}\right|^2\,\mathrm{d}t
\geq \frac{3\pi^2}{64}\sum_{k=1}^N|a_k|^2
$$
as soon as $T\geq 2$ while for $1<T<2$ a similar inequality remains true
provided one replaces $\dfrac{3\pi^2}{64}$
by $\dfrac{\pi^2}{8}\dfrac{T^2-1}{T^3}$. Note that Ingham \cite{In} proved that there is a sequence $(\lambda_k)$
for which the inequality fails for $T=1$.

Let us conclude this introduction with a word on the strategy of proof. This is closely related
to the one implemented by McGehee, Pigno and Smith as extended by
Nazarov to prove Theorem \ref{th:Naz}, but we here follow constants more closely. Further, we introduce various parameters which are optimized in the last step.
We fix a (non-harmonic) trigonometric polynomial
\begin{equation}
\label{eq:defphi}
\phi(t)=\sum_{k=1}^Na_ke^{2\pi i\lambda_kt}
\quad\mbox{and}\quad
S=\sum_{k=1}^N\frac{|a_k|}{k}.
\end{equation}
We then write $|a_k|=a_ku_k$ with $u_k$ complex numbers of modulus 1 and introduce
$$
U(t)=\sum_{k=1}^N\frac{u_k}{k}e^{2\pi i\lambda_kt}.
$$
Using the orthogonality relation \eqref{eq:ortho}
we see that
\begin{equation}
\label{eq:errorT_0}
S=\lim_{T\to+\infty}\frac{1}{T}\int_{-T/2}^{T/2}\phi(t)U(t)\ud t.
\end{equation}
The second step will consist in correcting $U$ into $V$ in such a way that $\|V\|_\infty\leq A$
where $A$ is a numerical constant (that does not depend on $N$ or $T$) and so that, for each $k$,
$$
\lim_{T\to+\infty}\frac{1}{T}\abs{\int_{-T/2}^{T/2}
\bigl(U(t)-V(t)\bigr)e^{-2i\pi\lambda_k t}\ud t}\leq\frac{\alpha}{k}
$$
with $\alpha<1$. In particular, if we multiply by $a_k$ and sum over $k$, we get 
$$
\lim_{T\to+\infty}\frac{1}{T}\abs{\int_{-T/2}^{T/2}\bigl(U(t)-V(t)\bigr)\phi(t)\ud t}\leq \alpha S.
$$
Then, writing
$$
S=\lim_{T\to+\infty}\frac{1}{T}\int_{-T/2}^{T/2}\phi(t)V(t)\ud t+
\lim_{T\to+\infty}\frac{1}{T}\int_{-T/2}^{T/2}\phi(t)\bigl(U(t)-V(t)\bigr)\ud t
$$
we would obtain
$$
S\leq \norm{V}_\infty\lim_{T\to+\infty}\frac{1}{T}\int_{-T/2}^{T/2}|\phi(t))|\ud t+\alpha S
$$
that is
$$
S\leq \frac{A}{1-\alpha}\lim_{T\to+\infty}\frac{1}{T}\int_{-T/2}^{T/2}|\phi(t))|\ud t
$$
as desired.

The difficulty in implementing this strategy lies in the fact that one must control $\phi,U,V$
over the entire real line. We will instead fix a large $T$ and use an auxiliary function adapted to 
$[-T/2,T/2]$ so as to only do the computations over this interval while controlling erros. Here we will exploit
the fact that $T$ is large that allows us to change Nazarov's auxiliary function into a better behaved one.
The first task is then to estimate the error made when replacing the limit in \eqref{eq:errorT_0}
with the mean over $[-T/2,T/2]$. The second step is then the correction
of $U$ into a bounded $V$. This correction is only made over the interval $[-T/2,T/2]$ and
is roughly done the same way as was originally done by McGehee, Pigno, Smith,
but implementing the improvements made by Stegeman and Yabuta and again controlling errors.

\medskip

\medskip

Let us conclude this introduction with a simple applcation.
Consider a curve in the complex plane of the form
$\dst \Gamma=\{z=P(t)=\sum_{k=1}^Na_ke^{2i\pi \lambda_k t}, t\in[0,T]\}$ where 
$\lambda_{k+1}\geq\lambda_k+1$. Figure \ref{fig} shows two such curves.
Note that $P'(t)=2i\pi\sum_{k=1}^Na_k\lambda_k e^{2i\pi \lambda_k t}$. It follows from Theorem
\ref{th:AA} that the length
of $\Gamma$ is lower bounded by
$$
\ell(\Gamma)=\int_0^T|P'(t)|\,\mbox{d}t\geq \frac{T}{20}\sum_{k=1}^N\frac{|\lambda_k||a_k|}{k}
$$
when $T\geq 72$.
\begin{figure}[h]
\includegraphics[scale=0.6,trim=70 100 50 100,clip]{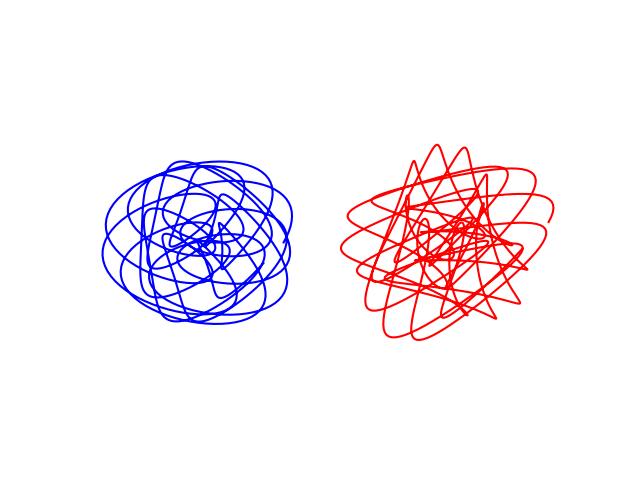}
\caption{Left: $1+e^{2i\pi t}+e^{20i t}$ and right $1+e^{2i\pi t}+e^{20i t}+e^{30i t}$
both for $t=0$ to $5$.}
\label{fig}
\end{figure}

%
%
%

\medskip

The remaining of the paper is devoted to the proof that is divided into three steps, a subsection beeing devoted to each.

\textcolor{blue}{For the convenience of the reader of this preprint, we have included the proof of Ingham's Inequality
and the proof by Hudson and Leckband in an appendix. Those are not meant to appear in the final version of the paper.}

\section{Proof of Theorem \ref{th:AA}}

\subsection{An auxiliary function and the estimate of $U$}\label{sec:22}

We now introduce several notations and parameters that will be fixed later:

-- a parameter $\delta\geq2$ and the sequence $(\beta_j)_{j\geq0}$ given by $\beta_0=1$, $\beta_{j+1}=\beta_j+\delta^j$
that is $\beta_j=\sum_{k=0}^{j-1}\delta^k=\frac{\delta^j-1}{\delta-1}$.
Up to enlarging $N$, we can assume that $N=\beta_{n+1}-1$ for some $n\geq 2$. We then define 
$$
\dd_j=\{k\in\Z\,:\ \beta_j\leq k<\beta_{j+1}\}
$$
so that $|\dd_j|=\delta^j$. Note that for every $\ell\in\{1,\ldots,N\}$
there is a unique $j_\ell\in\{1,\ldots,n\}$ such that $\ell\in\dd_{j_\ell}$.
Further, this allows to write $\sum_{k=1}^N$ in the form $\sum_{j=0}^n\sum_{k\in\dd_j}$.
Note also that if $k\in\dd_j$,
\begin{equation}
\label{eq:bound1/k+1}
\frac{1}{k+1}\leq \frac{1}{k+\frac{1}{\delta-1}}\leq (\delta-1)\delta^{-j}
\end{equation}

-- a sequence of real numbers $(\lambda_k)_{k=1,\ldots,N}$ 
such that for every $k,\ell$ $|\lambda_k-\lambda_\ell|\geq|k-\ell|$ (or equivalently
$\lambda_{k+1}\geq\lambda_k+1$ for $k=1,\ldots,N-1$);

-- a sequence of complex numbers $(a_k)_{k=1,\ldots,N}$ and we write $|a_k|=a_ku_k$ with 
$(u_k)_{k=1,\ldots,N}$ a sequence of complex numbers of modulus $1$;

-- an integer $p\geq 4$ and an interval $I_p=\ent{-\dfrac{p^2+p}{2},\dfrac{p^2+p}{2}}$
of length $|I_p|=p^2+p$.

We then define inductively $\ffi_1=\mathbf{1}_{[-p^2/2,p^2/2]}*\mathbf{1}_{[-1/2,1/2]}$ and 
$\ffi_{j+1}=\ffi_j*\mathbf{1}_{[-1/2,1/2]}$.
Note that, $\ffi_j$ is even, non-negative and $\norm{\ffi_j}_\infty\leq 1$ while 
$\norm{\ffi_j}_1=p^2$. We then define $\ffi=\dfrac{p^2+p}{p^2}\ffi_p$ so that $\ffi$ is supported in $I_p$, is bounded by $2$ and has Fourier transform
$$
\ff[\ffi](\lambda):=\int_{\R}\ffi(t)e^{-2i\pi\lambda t}\,\mbox{d}t=
(p^2+p)\frac{\sin p^2\pi\lambda}{p^2\pi\lambda}\left(\frac{\sin \pi\lambda}{\pi\lambda}\right)^p.
$$
We will mainly need that,
\begin{equation}
\label{eq:propffi}
\norm{\ffi}_\infty=\dfrac{p^2+p}{p^2}
\quad,\quad
\ff[\ffi](0)=|I_p|
\quad\mbox{and}\quad
|\ff[\ffi](\lambda)|\leq\dfrac{|I_p|}{(\pi\lambda)^p}.
\end{equation}

Finally, we will write
\begin{eqnarray*}
\phi(t)&=&\sum_{k=1}^N a_ke^{2i\pi\lambda_k t},\\
U(t)&=&\sum_{j=0}^n\frac{1}{|\dd_j|}\sum_{k\in\dd_j} u_ke^{-2i\pi\lambda_k t},\\
S&=&\sum_{j=0}^n\frac{1}{|\dd_j|}\sum_{k\in\dd_j}|a_k|.
\end{eqnarray*}

Note that in view of \eqref{eq:bound1/k+1},
\begin{equation}
\label{eq:SrightS}
\sum_{k=1}^N\frac{|a_k|}{k+1}\leq (\delta-1)S
\end{equation}
so that it is enough to bound $S$.

The following is the key estimate in this section: 

\begin{lemma} \label{l5}
With the previous notation, there is a $p(\delta)\geq 2$ such that, when $p\geq p(\delta)$ then
for $1\leq \ell \leq N$ and $j_\ell$ be the unique index for which $\ell\in\dd_{j_\ell}$, we have
\begin{equation}
\label{eq:l5}
\sum_{j=0}^n\frac{1}{|\dd_j|}\sum_{k\in\dd_j\setminus\{\ell\}}|\ff[\varphi](\lambda_{k}-\lambda_{\ell})|\leq \frac{1}{2}\,\delta^{-j_\ell}.
\end{equation}
\end{lemma}

\begin{proof}
Write
\begin{eqnarray*}
E&:=&\sum_{j=0}^n\frac{1}{|\dd_j|}\sum_{k\in\dd_j\setminus\{\ell\}}|\ff[\varphi](\lambda_{k}-\lambda_{\ell})|\notag\\
&=&\sum_{j=j_\ell-1}^{n}\frac{1}{|\dd_j|}\sum_{k\in\dd_j\setminus\{\ell\}}|\ff[\varphi](\lambda_{k}-\lambda_{\ell})|
+\sum_{j=0}^{j_\ell-2}\frac{1}{|\dd_j|}\sum_{k\in\dd_j}|\ff[\varphi](\lambda_{k}-\lambda_{\ell})|\notag\\
&=&E_++E_-.\label{eq:EE1E2}
\end{eqnarray*}

\smallskip

For the estimate of $E_+$, as $|\lambda_{k}-\lambda_{\ell}|\geq |k-\ell|$,
$|\ff[\varphi](\lambda_{k}-\lambda_{\ell})|\leq \dfrac{|I_p|}{\pi^{p}}\dfrac{1}{|k-\ell|^{p}}$ and
$$
    E_+\leq
    \frac{|I_p|}{\pi^p}\delta^{-(j_\ell-1)}\sum_{j=j_\ell-1}^{n}\sum_{k\in\dd_j\setminus\{\ell\}}\frac{1}{|k-\ell|^p}
    \leq \frac{2|I_p|}{\pi^p}\delta^{-(j_\ell-1)} \sum_{m=1}^{\infty}\frac{1}{m^p}
   \leq\frac{4|I_p|}{\pi^p}\delta^{-(j_\ell-1)}
$$
since $\delta>1$ and the last series is bounded by $\dst\sum_{m=1}^{\infty}\frac{1}{m^2}=\frac{\pi^2}{6}\leq 2$.
It remains to notice that 
$\dfrac{4\delta |I_p|}{\pi^p}=\dfrac{4\delta (p^2+p)}{\pi^p}\leq \dfrac{1}{4}$ when $p$ is large enough
to get
\begin{equation}
E_+\leq \frac{\delta^{-j_\ell}}{4}.
\label{eq:sumE1}
\end{equation}

\smallskip

For the second sum, note that it is only present when $j_\ell\geq 2$.
So, if $k\in\dd_j$ with $j\leq j_\ell-2$ then
$$
\ell-k\geq \beta_{j_\ell}-\beta_{j_\ell-1}=\delta^{j_\ell-1}
$$
thus 
$$
|\ff[\varphi](\lambda_{k}-\lambda_{j})|\leq\dfrac{|I_p|}{\pi^p \delta^{(j_\ell-1)p}}.
$$ 
It follows that
$$
E_-\leq\sum_{j=0}^{j_\ell-2}\sum_{k\in\dd_j}\dfrac{|I_p|}{\pi^p \delta^{(j_\ell-1)p}}
\leq \beta_{j_\ell-1}\dfrac{|I_p|}{\pi^p\delta^{(j_\ell-1)p}}
$$
since there are at most $\beta_{j_\ell-1}$ terms in this sum. But 
$\beta_{j_\ell-1}\leq \dfrac{\delta^{j_\ell-1}}{\delta-1}$
so that
\begin{equation}\label{N2}
E_-\leq \dfrac{|I_p|\delta}{\pi^p(\delta-1)\delta^{(j_\ell-1)(p-2)}}\delta^{-j_\ell}
\leq \frac{1}{4}\delta^{-j_\ell}
\end{equation}
when $p$ is large enough to have $\dfrac{|I_p|\delta}{\pi^p(\delta-1)}\leq\dfrac{1}{4}$
since $\delta^{(j_\ell-1)(p-2)}\geq 1$.

It remains to put \eqref{eq:sumE1}-\eqref{N2} into \eqref{eq:EE1E2} to obtain
the result.
\end{proof}

\begin{remark}
\label{rem:rqpdelta}
The proof shows that $p(\delta)$ is the smallest integer such that
$$
\frac{4\delta (p^2+p)}{\min(1,\delta-1)\pi^p}\leq \dfrac{1}{4}.
$$
For instance, if we choose $\delta=4$, we will obtain $p(\delta)=8$.
\end{remark}

From this, we deduce the following:

\begin{corollary}\label{l6}
With the notations above, for $\ell=1,\ldots,N$
\begin{equation}\label{e1}
   \abs{\frac{1}{|I_p|}\int_{I_p}U(t)e^{2i\pi\lambda_{\ell}t}\varphi(t)\,\mathrm{d}t-\frac{u_{\ell}}{\delta^{j_\ell}}}
   \leq \frac{1}{2|I_p|}\frac{1}{\delta^{j_\ell}}
\end{equation}
\end{corollary}

\begin{proof}
Indeed
\begin{eqnarray*}
\frac{1}{|I_p|}\int_{I_p}U(t)e^{2i\pi\lambda_{\ell}t}\varphi(t)\ud t
&=&\sum_{j=0}^n\frac{1}{|\dd_j|}\sum_{k\in\dd_j}u_k\frac{1}{|I_p|}\int_{I_p}e^{-2i\pi\lambda_{k}t}e^{2i\pi\lambda_{\ell}t}\varphi(t)\ud t\\
&=&\frac{1}{|I_p|}\sum_{j=0}^n\frac{1}{|\dd_j|}\sum_{k\in\dd_j}u_k\ff[\varphi](\lambda_{k}-\lambda_{\ell})\\
&=&u_\ell\delta^{-j_\ell}\frac{\ff[\varphi](0)}{|I_p|}
+\frac{1}{|I_p|}\sum_{j=0}^n\frac{1}{|\dd_j|}\sum_{k\in\dd_j\setminus\{\ell\}}u_k\ff[\varphi](\lambda_{k}-\lambda_{\ell}).
\end{eqnarray*}
As $\ff[\varphi](0)=|I_p|$, we get
$$
\abs{\frac{1}{|I_p|}\int_{I_p}U(t)e^{2i\pi\lambda_{\ell}t}\varphi(t)\ud t-\frac{u_{\ell}}{\delta^{j_\ell}}}
\leq\frac{1}{|I_p|}\sum_{j=0}^n\frac{1}{|\dd_j|}\sum_{k\in\dd_j\setminus\{\ell\}}|\ff[\varphi](\lambda_{k}-\lambda_{\ell})|
$$ 
and Proposition \ref{l5} gives the result.
\end{proof}

This allows us to obtain the approximation of $S$ by an integral of $\phi(t)U(t)$.

\begin{proposition}\label{cor:estSdelta}
Under the previous notation
$$
S\leq \frac{2|I_p|}{2|I_p|-1}\abs{\frac{1}{|I_p|}\int_{I_p}U(t)\phi(t)\ffi(t)\,\mathrm{d}t}.
$$
\end{proposition}

\begin{proof}
According to \eqref{e1},
$$
   \abs{\frac{1}{|I_p|}\int_{I_p}U(t)e^{2i\pi\lambda_{\ell}t}\varphi(t)\,\mathrm{d}t-u_\ell\delta^{-j_\ell}}
   \leq \frac{1}{2|I_p|}\delta^{-j_\ell}.
$$
Multiplying the expression in the absolute value by $a_\ell$, we get
$$
   \abs{\frac{1}{|I_p|}\int_{I_p}U(t)a_{\ell}e^{2i\pi\lambda_{\ell}t}\varphi(t)\,\mathrm{d}t-\frac{|a_\ell|}{\delta^{j_\ell}}}
   \leq \frac{|a_\ell|}{2|I_p|}\delta^{-j_\ell}.
$$
The triangle inequality then gives 
$$
\abs{\frac{1}{|I_p|}\int_{I_p}U(t)\phi(t)\ffi(t)\,\mathrm{d}t - S}\leq \frac{1}{2|I_p|} S.
$$
The result follows with the reverse triangular inequality.
\end{proof}

\subsection{Construction of $V$}\label{sec:23}

Before we start this section, let us recall Hilbert's inequality ({\it see e.g.} \cite[Chapter 10]{CQ}).

\begin{lemma}[Hilbert's inequality]\label{l7}
Let $\lambda_{1}, \ldots, \lambda_N$ be real numbers with $|\lambda_k-\lambda_\ell|\geq 1$ when $k \neq \ell$, 
and let $z_{1}, \ldots, z_N$ be complex numbers. We have
$$ 
\left| \sum_{\substack{1\leq k,\ell \leq N\\ k\neq \ell } } \frac{z_{k}\overline{z_{\ell}}}{\lambda_{k}-\lambda_{\ell}} \right|\leqslant \pi\sum_{k=1}^N|z_{k}|^{2}.
$$
\end{lemma}

We will now decompose $U$ into $\dd_j$-blocs $f_j$.
More precisely, we set
$$
\tilde f_j(t)=\sum_{k\in\dd_j} u_k e^{-2i\pi\lambda_k t},\quad
f_j=\frac{1}{|\dd_j|}\tilde f_j(t)
$$
so that
$$
U(t)=\sum_{j=0}^n\frac{1}{|\dd_j|}\sum_{k\in\dd_j} u_ke^{-2i\pi\lambda_k t}
=\sum_{j=0}^nf_j(t).
$$

Our aim in this section is to modify $U$ in such a way that we obtain a trigonometric polynomial
$V$ that is in a sense still similar to $U$ but satisfies an $L^\infty$ bound that is uniform in $N$.

We start by estimating the norms of the $f_j$'s:

\begin{lemma}\label{l8} With the above notation, we have
\begin{enumerate}
    \item $\|f_{j}\|_{L^{2}(I_p)}\leq \sqrt{|I_p|+1}\delta^{-j/2}$;
    \item $\|f_{j}\|_{\infty}\leq 1$.
\end{enumerate}
\end{lemma}

\begin{proof}
For the second bound, obviously $\|\tilde f_{j}\|_{\infty}\leq |\dd_j|=\delta^j$. For the first bound,
we have
\begin{eqnarray*}
\|\tilde f_{j}\|_{L^{2}(I_p)}^{2}
&=&\int_{I_p}\tilde f_{j}(t)\overline{\tilde f_{j}(t)}\ud t
=\int_{I_p}\sum_{k,\ell\in \dd_j}u_{k}\overline{u_{\ell}}e^{2i\pi(\lambda_{k}-\lambda_{\ell})t}\ud t\\
&=&|I_p|\sum_{k \in  \dd_{j}}|u_k|^2
+\sum_{\substack{k,\ell\in \dd_j\\ k\neq \ell}}
u_{k}\overline{u_{\ell}}\int_{-|I_p|/2}^{|I_p|/2}e^{2i\pi(\lambda_{k}-\lambda_{\ell})t}\ud t\\
&=&|I_p|\,|\dd_j|
+\sum_{\substack{k,\ell\in \dd_j\\ k\neq \ell}}u_k\overline{u_\ell}\left(
\frac{e^{i|I_p|\pi(\lambda_k-\lambda_\ell)}-e^{-i|I_p|\pi(\lambda_k-\lambda_\ell)}}{2i\pi(\lambda_k-\lambda_\ell)}
\right).
\end{eqnarray*}
Now, set $z_k^\pm=u_ke^{\pm i|I_p|\pi\lambda_k}$ so that $|z_k^\pm|=1$.
We have just shown that
$$
\|\tilde f_{j}\|_{L^2(I_p)}^{2}
=|I_p|\,|\dd_j|
+\frac{1}{2i\pi}\sum_{\substack{k,\ell\in \dd_j\\ k\neq \ell}}\frac{z_k^+\overline{z_\ell^+}}{\lambda_k-\lambda_\ell}
-\frac{1}{2i\pi}\sum_{\substack{k,\ell\in \dd_j\\ k\neq \ell}}\frac{z_k^-\overline{z_\ell^-}}{\lambda_k-\lambda_\ell}.
$$
Applying Hilbert's Inequality to the  last two sums, we get
$$
    \|\tilde f_{j}\|_{L^{2}(I_p)}^{2}\leq |I_p|\,|\dd_j|+\frac{1}{2} 
    \sum_{k\in \dd_{j}}|z_k^+|^{2}+\frac{1}{2} \sum_{k\in \dd_{j}}|z_k^-|^2
    =(|I_p|+1)|\dd_j|.
$$
The bound for $ \|f_{j}\|_{L^{2}(I_p)}^2$ follows.
\end{proof}

Note that the proof also shows that $\|f_{j}\|_{L^{2}(I_p)}\geq \sqrt{|I_p|-1}\delta^{-j/2}$.
\begin{notation}
For a function $F\in L^2(I_p)$ and $s\in\Z$, we write
$$
c_s^p(F)=\frac{1}{|I_p|}\int_{I_p}F(t)e^{-2i\pi\frac{st}{|I_p|}}\ud t
$$
for the Fourier coefficients of $F$. Its Fourier series is then
$$
F(t)=\sum_{s\in\Z} c_s^p(F)e^{2i\pi\frac{st}{|I_p|}}
$$
and Parseval's relation reads
$$
\frac{1}{|I_p|}\int_{I_p}|F(t)|^2\,\mathrm{d}t=\sum_{s\in\Z} |c_s^p(F)|^2.
$$
\end{notation}

We then write the Fourier series of each $|f_j|\in L^2(I_p)$ as
$$
|f_{j}|=\sum_{s\in \mathbb{Z}}c_s^p(|f_j|)e^{2i\pi\frac{st}{|I_p|}}
$$
to which we associate $h_{j}\in L^{2}(I_p)$ defined via its Fourier series as
$$
h_j(t)=c_0^p(|f_j|)+2\sum_{s=1}^{\infty}c_s^p(|f_j|)e^{2i\pi\frac{st}{|I_p|}}.
$$ 

\begin{lemma}\label{l9} For $0\leq j \leq n$, the following  properties hold
\begin{enumerate}
    \item $\mathrm{Re}(h_{j})=|f_j|\leqslant 1$;
    \item $\|h_{j}\|_{L^2(I_p)}\leq \sqrt{2}\|f_{j}\|_{L^2(I_p)}$.
\end{enumerate}
\end{lemma}

\begin{proof} First, as $|f_j|$ is real valued,
$c_0^p(|f_j|)$ is also real, and $\overline{c_s^p(|f_j|)}=c_{-s}^p(|f_j|)$ for every $s\geq 1$. 
A direct computation then shows that $\mathrm{Re}(h_{j})=|f_{j}|$ which is less than 1 by lemma \ref{l8} while
Parseval shows that $\|h_{j}\|_{2}\leqslant \sqrt{2}\|f_{j}\|_{2}$.
\end{proof}

We now define a sequence $(F_{j})_{j=0,\ldots,n}$ inductively through
$$
F_0= f_0
\quad\mbox{and}\quad
F_{j+1}=F_{j}e^{-\eta h_{j+1}}+f_{j+1}
$$
where $0<\eta\leq 1$ is a real number that we will fix later. Further set
$$
E_{\eta}:=\sup_{0 <x \leq 1} \dfrac{x}{1-e^{-\eta x}}=\frac{1}{\eta}\sup_{0 <x \leq \eta} \dfrac{x}{1-e^{- x}}
=\frac{1}{1-e^{-\eta}}.
$$
\begin{lemma}\label{l10}
For $0\leq j\leq n$, $\| F_{j}\|_{\infty} \leq E_{\eta}$.
\end{lemma}

\begin{proof}
By definition of $E_\eta$, if $C\leq E_{\eta}$ and
$0\leq x\leq 1$, then  $Ce^{-\eta x}+x\leq E_{\eta} e^{-\eta x}+x\leq E_{\eta}$.

We can now prove by induction over $j$ that $|F_j|\leq E_{\eta}$ from which the lemma follows.
First, when $j=0$, from Lemma \ref{l8} we get
$$
\|F_0\|_{\infty}=\|f_0\|_{\infty} \leq 1\leq E_{\eta}.
$$
Assume now that $\|F_j\|_\infty\leq E_{\eta}$, then
\begin{eqnarray*}
|F_{j+1}(t)|&=&|F_{j}(t)e^{-\eta h_{j+1}(t)}+ f_{j+1}(t)|\leq
|F_{j}(t)|e^{-\eta\Re\bigl(h_{j+1}(t)\bigr)}+|f_{j+1}(t)|\\
&=&|F_{j}(t)|e^{-\eta|f_{j+1}(t)|}+|f_{j+1}(t)|.
\end{eqnarray*}
As $|f_{j+1}(t)|\leq 1$ and $|F_{j}(t)|\leq E_{\eta}$, we get
$|F_{j+1}(t)|\leq E_{\eta}$ as claimed.
\end{proof}

\begin{lemma}\label{l11}
For $0\leq \ell\leq n$ and $j=0,\ldots,k$, let $g_{j,k}=e^{-\eta H_{j,k}}$
with
$$
H_{j,k}=\begin{cases}
h_{j+1}+\ldots +h_{k} & \mbox{when } j<k\\
0  & \mbox{when }j=k
\end{cases}.
$$
Then $F_{k}=\displaystyle\sum_{j=0}^{k}f_{j}g_{j,k}$.
Moreover 
$$
\| H_{j,k} \|_{L^2(I_p)}\leq\frac{\sqrt{2(|I_p|+1)}}{\sqrt{\delta}-1}\delta^{-j/2}.
$$
\end{lemma}

\begin{proof} For the first part, we use induction on $k$.
First, when $k=0$, $H_{0,0}=0$ thus $g_{0,0}=1$ and, indeed, we have
$$
F_0=f_0=f_{0}g_{0,0}.
$$ 
Assume now that the formula has been established at rank $k-1$
and let us show that $F_{k}=\displaystyle\sum_{j=0}^{k}f_{j}g_{j,k}$.
By construction, we have
$$
F_{k}=F_{k-1}e^{-\eta h_{k}}+ f_{k}
=\left(\sum_{j=0}^{k-1}f_{j}g_{j,k-1}\right)e^{-\eta h_{k}}+f_{k}.
$$
with the induction hypothesis. It remains to notice that $g_{k,k}=e^{-\eta H_{k,k}}=1$
and that, for $j=0,\ldots,k-1$, $H_{j,k}=H_{j,k-1}+h_k$ thus
$g_{j,k}=g_{j,k-1}e^{-\eta h_{k}}$ so that, indeed, we have
$F_k=\displaystyle\sum_{j=0}^{k}f_{j}g_{j,k}$.

\smallskip

Next, it is enough to estimate $H_{j,k}$ when $j<k$ in which case
$$
\|H_{j,k}\|_{L^2(I_p)}\leq \sum_{r=j+1}^{k}\|h_r\|_{L^2(I_p)}
\leq \sqrt{2}\sum_{r=j+1}^{k}\|f_{r}\|_{L^2(I_p)}
$$
with Lemma \ref{l9}. But then, from Lemma \ref{l8} we get
$$
\|H_{j,k}\|_{L^2(I_p)}\leq\sqrt{2(|I_p|+1)}\sum_{r=j+1}^{k}\delta^{-\frac{r}{2}}
\leq \sqrt{2(|I_p|+1)}\frac{\delta^{-\frac{j+1}{2}}}{1-\delta^{-1/2}}
=\frac{\sqrt{2(|I_p|+1)}}{\sqrt{\delta}-1}\delta^{-j/2}
$$
as claimed.
\end{proof}

\begin{lemma}\label{l13}
Let $0\leq k\leq n$ and $0\leq j\leq k$, then
\begin{enumerate}
    \item the negative Fourier coefficients of $g_{j,k}(t)-1$  vanish so that its Fourier series writes
    $$
   g_{j,k}(t)-1=\displaystyle \sum_{s \geqslant 0}  c^p_s(g_{j,k}-1)e^{i\pi\frac{st}{|I_p|}};
    $$
    \item $\dst\|g_{j,k}-1\|_{L^2(I_p)}
    \leq \eta \frac{\sqrt{2(|I_p|+1)}}{\sqrt{\delta}-1}\delta^{-j/2}$;
    \end{enumerate}
\end{lemma}

\begin{proof} When $j=k$, $g_{k,k}(t)-1=0$ and there is nothing to do.
When $j<k$, $\Re(H_{j,k})=\dst\sum_{r=j+1}^k\Re(h_r)\geq 0$ so that
this is \cite[Lemma, p614]{MPS} which shows the first statement and that
$$
\|g_{j,k}-1\|_{L^2(I_p)}
:=\|e^{-\eta H_{j,k}}-1\|_{L^2(I_p)}\leq \eta\|H_{j,k}\|_{L^2(I_p)}.
$$
We then conclude with Lemma \ref{l11}.
\end{proof}

Recall that 
$$
U(t)=\sum_{j=0}^nf_{j}
$$
and we set 
$$
V^\eta=\dst F_{n}=\sum_{j=0}^{n}f_jg_{j,n}
$$
where the dependence on $\eta$ comes from
the definition of the $g_{j,n}$'s. In particular, 
\begin{equation}
\label{eq:boundV}
\|V^\eta\|_\infty\leq E_\eta.
\end{equation}

The key estimate here is the following:

\begin{proposition}\label{l14}  
Let $0<\eps\leq 1$, $N\geq 1$ and $\delta\geq e$
then there exists $P$ such that, if $p\geq P$,
there exists $\eta=\eta(p)\in (0,1)$ such that,
for $1\leq \ell \leq N$ and $j_\ell$ the unique index for which $\ell \in \dd_{j_\ell}$
\begin{equation}\label{e2}
    \left| \frac{1}{|I_p|}\int_{I_p}\bigl(U(t)-V^\eta(t)\bigr)e^{2i\pi\lambda_{\ell} t}\varphi(t)\ud t \right| 
    \leq \eps\delta^{-j_{\ell}}.
\end{equation}
Moreover, when $p\to+\infty$, $\eta(p)\to \eta_\infty=\dfrac{(\delta-1)(\sqrt{\delta}-1)}{\sqrt{2}\delta}\eps$.
\end{proposition}

\begin{proof}
To simplify notation, we write $g_j=g_{j,n}$ and $V=V^\eta$.
Then
\begin{eqnarray*}
    R&:=&\frac{1}{|I_p|}\int_{I_p}\bigl(U(t)-V(t)\bigr)e^{2i\pi \lambda_{\ell}t}\varphi(t)\ud t\\
    &=&\frac{1}{|I_p|}\int_{I_p}\sum_{0\leq  j \leq j_{\ell}-2}    f_{j}(g_{j}-1)e^{2i\pi\lambda_{\ell}t}\varphi(t)\ud t
    +\frac{1}{|I_p|}\int_{I_p}f_{j_{\ell}-1}(g_{j_{\ell}-1}-1)e^{2i\pi\lambda_{\ell}t}\varphi(t)\ud t\\
&& +\frac{1}{|I_p|}\int_{I_p}\sum_{j_{\ell} \leq j \leq n}f_{j}(g_{j}-1)e^{2i\pi\lambda_{\ell}t}\varphi(t)\ud t\\
    &:=&R_{-}+R_0+R_{+}.
\end{eqnarray*}

Let us first bound $R_{-}$ (which is only present when $j_{\ell}\geq 2$). For this, notice that if $s\in \mathbb{Z},$
\begin{eqnarray*}
\int_{I_p}\tilde f_{j}(t)\varphi(t)e^{2i\pi\lambda_{\ell}t} e^{-2i\pi\frac{s}{|I_p|}t}\ud t
&=&\int_{I_p}\sum_{r\in \dd_{j}}u_{r}\varphi(t)e^{-2i\pi (\lambda_{r}-\lambda_{\ell}+\frac{s}{|I_p|})}\ud t\\
&=&\sum_{r\in \dd_{j}}u_{r}    \ff[\varphi]\left(\lambda_{r}-\lambda_{\ell}+\frac{s}{|I_p|}\right).
\end{eqnarray*}
It follows that
\begin{eqnarray*}
\int_{I_p}\tilde f_{j}(g_{j}-1)e^{2i\pi \lambda_{\ell}t}\varphi(t)\ud t 
&=&\int_{I_p}\tilde f_{j}(t)\varphi(t)e^{2i\pi\lambda_{\ell}t}\sum_{s \geq 0}c_s^p(g_j-1)e^{2i\pi\frac{st}{|I_p|}}\ud t\\
&=&\sum_{s= 0}^{+\infty}c_s^p(g_j-1)\int_{I_p}\tilde f_{j}(t)\varphi(t)
e^{2i\pi\lambda_{\ell}t+\frac{2i\pi st}{|I_p|}}\ud t \\
&=&\sum_{s= 0}^{+\infty}c_s^p(g_j-1)\sum_{r\in \dd_{j}}u_{r}
\ff[\varphi]\left(\lambda_{r}-\lambda_{\ell}-\frac{s}{|I_p|}\right)\\
&=&\sum_{r\in \dd_{j}}u_{r}
\sum_{s=0}^{\infty}c_s^p(g_j-1)\ff[\varphi]\left(\lambda_{r}-\lambda_{\ell}-\frac{s}{|I_p|}\right).
\end{eqnarray*}
Finally, we get
$$
R_{-}=\frac{1}{|I_p|}\sum_{0\leq j \leq j_\ell-2}\frac{1}{|\dd_j|}\sum_{r\in \dd_{j}}u_{r}
\sum_{s=0}^{\infty}c_s^p(g_j-1)\ff[\varphi]\left(\lambda_{r}-\lambda_{\ell}-\frac{s}{|I_p|}\right).
$$
As 
$$
\left(\sum_{s\geq 0}|c_s^p(g_j-1)|^2\right)^{1/2}=\frac{1}{\sqrt{|I_p|}}\|g_j-1\|_{L^2(I_p)}\leq\eta 
\frac{\sqrt{2(|I_p|+1)/|I_p|}}{\sqrt{\delta}-1}
$$
we get with Cauchy-Schwarz
$$
|R_-|\leq \eta \frac{\sqrt{2(|I_p|+1)/|I_p|}}{\sqrt{\delta}-1}\sum_{0\leq j \leq  j_\ell-2}\frac{1}{|\dd_j|}
\sum_{r\in \dd_{j}}
\left(\sum_{s=0}^{\infty}\frac{1}{|I_p|^2}\abs{\ff[\varphi]\left(\lambda_{r}-\lambda_{\ell}-\frac{s}{|I_p|}\right)}^2\right)^{1/2}.
$$

We will first estimate the inner most sum. We will use the following simple estimate valid for $u\geq 1$:
\begin{eqnarray*}
\sum_{s=0}^{+\infty}\frac{1}{(s+u)^{2p}}=\frac{1}{u^{2p}}+\sum_{s=1}^{+\infty}\frac{1}{(s+u)^{2p}}
&\leq& \frac{1}{u^{2p}}+\int_{s=0}^{+\infty}\frac{\mbox{d}s}{(s+u)^{2p}}=\frac{1}{u^{2p}}+\frac{1}{(2p-1)u^{2p-1}}\\
&\leq& \frac{2}{u^{2p-1}}.
\end{eqnarray*}
Here $u=|I_p|(\lambda_{\ell}-\lambda_{r})$. 

First, for $0\leq j\leq j_\ell-2$, if $r\in\dd_j$ 
then $r< \beta_{j+1}$ and $\beta_\ell\leq \ell <\beta_{\ell+1}$
we first get
$$
|I_p|(\lambda_{\ell}-\lambda_{r})\geq |I_p|(\beta_\ell-\beta_{j+1})=|I_p|\frac{\delta^{j_\ell}-\delta^{j+1}}{\delta-1}\geq \frac{|I_p|}{2}\delta^{j_\ell -1}
$$
for all $0 \leq j \leq j_\ell -2 $. This then implies that
\begin{eqnarray*}
E_r:=\sum_{s=0}^{\infty}\frac{1}{|I_p|^2}\abs{\ff[\varphi]\left(\lambda_{k}-\lambda_{r}+\frac{s}{|I_p|}\right)}^2
&\leq&\sum_{s=0}^{\infty}\frac{1}{\left[\pi\left(\lambda_{k}-\lambda_{r}+\frac{s}{|I_p|}\right]\right)^{2p}}\\
&=&\left(\frac{|I_p|}{\pi}\right)^{2p}\sum_{s=0}^{\infty}\frac{1}{\bigl(|I_p|(\lambda_{k}-\lambda_{r})+s\bigr)^{2p}}\\
&\leq& \textcolor{black}{\frac{2^{2p} |I_p|}{\pi^{2p}}\delta^{-(j_\ell-1)(2p-1)}}
\end{eqnarray*}
It follows that
\textcolor{black}
{
$$
|R_-|\leq \eta \frac{\sqrt{2(|I_p|+1)}}{\sqrt{\delta}-1}
\sum_{0\leq j \leq j_\ell-2}\frac{\delta^{-(j_\ell-1)(p-1/2)}}{(\pi/2)^{p}}
=\eta \frac{\delta\sqrt{2(|I_p|+1)}}{(\pi/2)^{p}(\sqrt{\delta}-1)} (j_\ell-1)\delta^{-(j_\ell-1)(p-3/2)}
\delta^{-j_\ell}.
$$
}
Now, simple calculus shows that, when \textcolor{black}{$a\geq \dfrac{1}{2}$, $x\geq 1$ $(x-1)e^{-a(x-1)}\leq xe^ae^{-ax}$ which is  decreasing (in $x$).
 Thus (as $x\geq 2$ in our case)
 $$
 (j_\ell-1)\delta^{-(j_\ell-1)(p-3/2)}\leq \dfrac{2}{e^{(p-\frac{3}{2})
}}
 $$
 }
 leading to the bound
 \textcolor{black}
 {
\begin{equation}
\label{eq:boundr-}
|R_-|\leq \eta \frac{2\delta\sqrt{2(|I_p|+1)}}{e^{(p-\frac{3}{2})}(\pi/2)^{p}(\sqrt{\delta}-1)}\delta^{-j_\ell}.
\end{equation}
}
It is crucial for the sequel to note that we can write this as
$$
|R_-|\leq \eta\mu^-_p\delta^{-j_\ell}\qquad\mbox{with }\mu^-_p=\mu_p^-(\delta)\to0\mbox{ when }p\to+\infty.
$$

On the other hand, if $j=j_\ell-1$, $r\in\dd_j$ then 
$|I_p|(\lambda_{\ell}-\lambda_{r})\geq |I_p|$
and the same computation gives
$\dst E_r\leq \frac{2|I_p|}{\pi^{2p}}$. Repeating the computation of $R_-$
gives $|R_0|\leq \dst\eta \frac{2\sqrt{|I_p|+1}}{\pi^p}\frac{1}{\sqrt{\delta}-1}$.
We write this in the form
$$
|R_0|\leq \eta\delta^{-j_\ell}\mu_p^0(\delta,n)
$$
where $\mu_p^0(\delta,n)\to 0$ when $p\to+\infty$.

We can now estimate $R_+$.
\begin{eqnarray}
|R_+| &\leq& \sum_{j_\ell\leq j \leq n}
\frac{1}{|I_p|}\int_{I_p}|f_j(t)|\,|g_{j}(t)-1|\,|\ffi(t)|\ud t\label{eq:jell-1}\\
&\leq& \|\varphi\|_{\infty}\sum_{j_\ell \leq j \leq n} \frac{1}{|I_p|}\|f_j\|_{L^2(I_p)}
\|g_j-1\|_{L^2(I_p)}\notag\\
&\leq&\eta \frac{\sqrt{2}}{\sqrt{\delta}-1}
\frac{p^2+p}{p^2}\frac{|I_p|+1}{|I_p|}\sum_{j_\ell \leq j \leq n}\delta^{-j}\notag\\
&\leq&\eta\frac{\sqrt{2}\delta}{(\sqrt{\delta}-1)(\delta-1)}\frac{p^2+p}{p^2}\frac{|I_p|+1}{|I_p|}
\delta^{-j_\ell}\notag\\
&=&\eta\frac{\sqrt{2}\delta}{(\sqrt{\delta}-1)(\delta-1)}\frac{p^2+p+1}{p^2}
\delta^{-j_\ell}\notag
\end{eqnarray}
We write this in the form
$$
|R_+|\leq \eta\left(\frac{\sqrt{2}\delta}{(\sqrt{\delta}-1)(\delta-1)}+\mu_p^{+}\right)\delta^{-j_\ell}
\quad\mbox{with }\lim_{p\to+\infty}\mu_p^+=0
$$
and conclude that $|R|\leq \eta\left(\frac{\sqrt{2}\delta}{(\sqrt{\delta}-1)(\delta-1)}+\mu_p\right)\delta^{-j_\ell}$
with $\mu_p=\mu_p^{-}(\delta)+\mu_p^{0}(\delta,n)+\mu_p^{+}\to 0$ (depending on $\delta$ and $n$
thus $N$).
\end{proof}

\begin{remark}
An inspection of the proof shows that the dependence of $P$ on $N$ only comes from $R_0$.
This is harmless when we let $p\to+\infty$ which then implies $|I_p|\to+\infty$
{\it i.e.} when we prove Theorem \ref{th:AA} {\it i) \& ii)} but is not possible
when proving {\it iii)}. To avoid that issue, one
can then bound $R_-$ and $R_0+R_+$ instead of $R_-+R_0$ and $R_+$.
The same proof works but the price to pay are slightly worse constants:
 \textcolor{black}
 {
$$
|R_-|\leq \eta \frac{2\delta\sqrt{2(|I_p|+1)}}{e^{(p-\frac{3}{2})}(\pi/2)^{p}(\sqrt{\delta}-1)}\delta^{-j_\ell}.
$$
}
and
$$
|R_0+R_+|\leq\eta\frac{\sqrt{2}\delta^2}{(\sqrt{\delta}-1)(\delta-1)}\frac{p^2+p}{p^2}\frac{|I_p|+1}{|I_p|}
\delta^{-j_\ell}
$$
since we can include $R_0$ into the sum \eqref{eq:jell-1} defining $R_+$
by starting it at $j_\ell-1$ instead of $j_\ell$. The consequence is that the $\delta$
on the numerator of the bound of $R_+$ becomes $\delta^2$.

But then
$$
|R|\leq\eta\left(\textcolor{black}{\frac{2\delta\sqrt{2(|I_p|+1)}}{e^{(p-\frac{3}{2})}(\pi/2)^{p}(\sqrt{\delta}-1)}}+\frac{\sqrt{2}\delta^2}{(\sqrt{\delta}-1)(\delta-1)}\frac{p^2+p+1}{p^2}\right)
\delta^{-j_\ell}.
$$
Taking any
\textcolor{black}
{
\begin{eqnarray*}
\eta &\leq& \frac{\eps}{\frac{2\delta \sqrt{2(|I_p|+1)}}{e^{(p-\frac{3}{2})}(\pi/2)^p(\sqrt{\delta}-1)}+\frac{\sqrt{2}\delta^2}{(\sqrt{\delta}-1)(\delta-1)}\frac{p^2+p+1}{p^2}}\\
&=& \frac{\eps(\sqrt{\delta}-1)(\delta-1)}{\frac{2\delta \sqrt{2(|I_p|+1)}(\delta-1)}{e^{(p-\frac{3}{2})}(\pi/2)^p}+\sqrt{2}\delta^2\frac{p^2+p+1}{p^2}}
\end{eqnarray*}
}
will then still give \eqref{e2}, provided this $\eta$ satisifies $0<\eta<1$. 

However, this quantity is too complicated to hope to be able to handle it in an optimisation process.
Instead, we will determine the condition on $p$ for the parameter that almost optimise the case $p\to+\infty$,
namely $\eps=\frac{1}{2}$ and $\delta=4$. In this case, the smallest possible $p$ in
the first part is $p=8$. One can then do a computer check to see
that, when $\eps=1/2$ and $p\geq 8$, \eqref{e2} is valid for $\eta=0.058$.
\end{remark}

\begin{corollary}
Under the conditions of Proposition \ref{l14} 
\begin{equation}\label{e3}
    \abs{\frac{1}{|I_p|}\int_{I_p}\bigl(U(t)-V^\eta(t)\bigr)\phi(t)\varphi(t)\ud t}
    \leq \eps S.
\end{equation}
Further, if $\eta=0.058$ and $\delta=4$, then for $p\geq 8$, \eqref{e3} holds for $\eps=1/2$.
\end{corollary}

\begin{proof}
As $\phi(t)=\sum_{k=1}^N a_ke^{2i\pi\lambda_k t}$, it suffices to use the triangular inequality and \eqref{e2}.
\end{proof}

\subsection{End of the proof}

The end of the proof consists in applying first Proposition \ref{cor:estSdelta}
$$
S\leq \frac{2|I_p|}{2|I_p|-1}\abs{\frac{1}{|I_p|}\int_{I_p}
U(t)\phi(t)\ffi(t)\,\mathrm{d}t}.
$$
Then, we fix an $0<\eps<\frac{2|I_p|-1}{2|I_p|}$ and take $\eta<\eta(p)$ as in Proposition \ref{l14}
and apply \eqref{e3} to get
\begin{eqnarray*}
S&\leq& \frac{2|I_p|}{2|I_p|-1}\abs{\frac{1}{|I_p|}\int_{I_p}\bigl(U(t)-V^\eta(t)\bigr)\phi(t)\ffi(t)\,\mathrm{d}t}
+\frac{2|I_p|}{2|I_p|-1}\abs{\frac{1}{|I_p|}\int_{I_p}V^\eta(t)\phi(t)\ffi(t)\,\mathrm{d}t}\\
&\leq& \frac{2|I_p|}{2|I_p|-1}\eps S+ \frac{2|I_p|}{2|I_p|-1}\norm{V^\eta}_\infty\norm{\ffi}_\infty
\frac{1}{2|I_p|}\int_{I_p}|\Phi(t)|\,\mathrm{d}t\\
&\leq&\frac{2|I_p|}{2|I_p|-1}\eps S +\frac{p^2+p}{p^2}\frac{2|I_p|}{2|I_p|-1}E_{\eta}
\frac{1}{|I_p|}\int_{I_p}|\Phi(t)|\,\mathrm{d}t.
\end{eqnarray*}

We thus obtain from \eqref{eq:final} that 
\begin{equation}
\label{eq:final}
\left(1-\frac{2|I_p|}{2|I_p|-1}\eps\right)S\leq \frac{p^2+p}{p^2}\frac{2|I_p|}{2|I_p|-1}E_{\eta}
\frac{1}{|I_p|}\int_{I_p}|\Phi(t)|\,\mathrm{d}t.
\end{equation}

Taking $\delta=4$, $\eps=1/2$, $p\geq 8$ so that $T=|I_p|\geq 72$
and $\eta=0.058$, \eqref{eq:final} reads
$$
S\leq\frac{2\times(72)^2}{71\times64}\frac{1}{1-e^{-0.058}}
\frac{1}{T}\int_{-T/2}^{T/2}|\Phi(t)|\,\mathrm{d}t.
$$
From \eqref{eq:bound1/k+1}, we finally get
$$
\sum_{k=1}^N\frac{|a_k|}{k+1}\leq 3\times \frac{(144)^2}{64\times 142}\frac{1}{1-e^{-0.058}}
\frac{1}{T}\int_{-T/2}^{T/2}|\Phi(t)|\,\mathrm{d}t
\leq \frac{122}{T}\int_{-T/2}^{T/2}|\Phi(t)|\,\mathrm{d}t
$$
establishing Theorem \ref{th:AA}{\it iii)}. 

\smallskip

We will now establish Theorem \ref{th:AA}{\it i)}. To do so, let $p\to+\infty$ in \eqref{eq:final} to get
\begin{equation}
\label{eq:epsinfini}
(1-\eps)S\leq \frac{1}{1-e^{-\eta}}
\lim_{T\to+\infty}\frac{1}{T}\int_{-T/2}^{T/2}|\Phi(t)|\,\mathrm{d}t.
\end{equation}

This relation is valid for every $\delta>1$, every $0<\eps<1$
and every 
$$
\eta<\eta_\infty=\dfrac{(\delta-1)(\sqrt{\delta}-1)}{\sqrt{2}\delta}\eps.
$$
By continuity, we can thus replace $\eta$ by $\eta_\infty$ in \eqref{eq:epsinfini}.
Further, we may use again \eqref{eq:bound1/k+1} to get
\begin{equation}
\label{eq:optimisation}
\sum_{k=1}^N\frac{|a_k|}{k+1}\leq \frac{\delta-1}{(1-\eps)\bigl(1-e^{-\eta_\infty}\bigr)}
\lim_{T\to+\infty}\frac{1}{T}\int_{-T/2}^{T/2}\abs{\sum_{j=1}^N a_je^{2i\pi\lambda_j t}}\,\mathrm{d}t.
\end{equation}

It remains to choose the parameters $\delta$ and $\eps$
so as to minimize the factor of the Besikovitch norm.
A computer search shows that
$$
\frac{\delta-1}{(1-\eps)\bigl(1-e^{-\eta_\infty}\bigr)}
=\frac{\delta-1}{(1-\eps)\bigl(1-e^{-\frac{(\delta-1)(\sqrt{\delta}-1)}{\sqrt{2}\delta}\eps}\bigr)}
$$
takes its minimal value $\sim 25.1624$ for some $\eps,\delta$ with $0.4768\leq\eps\leq 0.4772$ and $3.70\leq
\delta \leq 3.75$. This
gives the claimed inequality
$$
\sum_{k=1}^N\frac{|a_k|}{k+1}\leq 25.2
\lim_{T\to+\infty}\frac{1}{T}\int_{-T/2}^{T/2}\abs{\sum_{k=1}^N a_ke^{2i\pi\lambda_k t}}\,\mathrm{d}t.
$$

\smallskip

A somewhat better estimate is possible when $|a_k|\geq1$. Indeed, in this case we proved in 
\eqref{eq:epsinfini} that
$$
n+1\leq \frac{1}{(1-\eps)\bigl(1-e^{-\frac{(\delta-1)(\sqrt{\delta}-1)}{\sqrt{2}\delta}\eps}\bigr)}
\lim_{T\to+\infty}\frac{1}{T}\int_{-T/2}^{T/2}\abs{\sum_{j=1}^N a_je^{2i\pi\lambda_j t}}\,\mathrm{d}t.
$$
But $n$ was defined as $N=\beta_{n+1}=\dst\frac{\delta^{n+1}-1}{\delta -1}$
that is 
$$
n+1=\frac{\ln\bigl(1+(\delta-1){N}\bigr)}{\ln\delta}
$$
since $\delta\geq 2$. We thus have
$$
\ln [1+(\delta-1)N]\leq \frac{\ln \delta}{(1-\eps)\left(1-e^{-\frac{(\delta-1)(\sqrt{\delta}-1)}{\sqrt{2}\delta}\eps}\right)}
\lim_{T\to+\infty}\frac{1}{T}\int_{-T/2}^{T/2}\abs{\sum_{j=1}^N a_je^{2i\pi\lambda_j t}}\,\mathrm{d}t.
$$
and we are looking for $\eps,\delta$ that minimize
$$
\frac{\ln \delta}{(1-\eps)\bigl(1-e^{-\frac{(\delta-1)(\sqrt{\delta}-1)}{\sqrt{2}\delta}\eps}\bigr)}
$$
and for the value of this minimum. The best value we obtain is $7.714$ for $\eps=0.28$
and $\delta=89.254$ which is essentially the same value as in \cite{Steg}
(note that $1/7.714=0.1296$ is the value given there).

We then obtain the following: if $|a_n|\geq 1$, then
$$
\lim_{T\to+\infty}\frac{1}{T}\int_{-T/2}^{T/2}\abs{\sum_{k=1}^N a_ne^{2i\pi\lambda_k t}}\,\mathrm{d}t
\geq 0.1296 \ln (1+88N).
$$

\subsection{Further comments}
\label{seccomments}

Let us make a few comments.

First, the inequality for a fixed $T$ implies the inequality for the Besikovitch norm.
This follows from a simple trick already used for Ingham's inequality. Indeed, once we establish
that 
$$
C_0 \int_{-T_0/2}^{T_0/2}\abs{\sum_{k=1}^N a_ke^{2i\pi\lambda_k t}}\,\mathrm{d}t\geq
\sum_{k=1}^N\frac{|a_k|}{k+1}
$$
for some $T_0>0$ then, changing varible $t=t_0+s$ we also have
$$
C_0 \int_{t_0-T_0/2}^{t_0+T_0/2}\abs{\sum_{k=1}^N a_ke^{2i\pi\lambda_k t}}\,\mathrm{d}t
=C_0\int_{-T_0/2}^{T_0/2}\abs{\sum_{k=1}^N a_ke^{2i\pi\lambda_k t_0}e^{2i\pi\lambda_k s}}\,\mathrm{d}s\geq
\sum_{k=1}^N\frac{|a_k|}{k+1}.
$$
Next, for any integer $M$, covering $[-MT_0,M T_0]$ by $2M$ intervals of length $T_0$, we get
$$
C_0\int_{-MT_0}^{MT_0}\abs{\sum_{k=1}^N a_ke^{2i\pi\lambda_k t}}\,\mathrm{d}t\geq
2M\sum_{k=1}^N\frac{|a_k|}{k+1}.
$$
Dividing by $2M$ and letting $M\to+\infty$ we get
$$
C_0T_0\lim_{T\to+\infty}\frac{1}{T}\int_{-T/2}^{T/2}\abs{\sum_{k=1}^N a_ke^{2i\pi\lambda_k t}}\,\mathrm{d}t\geq \sum_{k=1}^N\frac{|a_k|}{k+1}.
$$
In particular, Nazarov's result also implies the result of Hudson and Leckband
(but with worse constants).

\section*{Acknoledegments}

The authors wish to thank E. Lyfland for bringing Reference \cite{TB} to their attention
and the anonymous referee for suggestions that lead to improvements of the manuscript.

\section{Data availability}
No data has been generated or analysed during this study.

\section{Funding and/or Conflicts of interests/Competing interests}

The authors have no relevant financial or non-financial interests to disclose.

Ce travail a bénéficié d'une aide de l'\'Etat attribu\'e \`a l'Universit\'e de Bordeaux en tant qu'Initiative d'excellence, au titre du plan France 2030.

The research of K.K. is supported by the project ANR-
18-CE40-0035 and by the Joint French-Russian Research Project PRC CNRS/RFBR 2017–
2019

\appendix

\section{The proof by Hudson and Leckband}
This appendix is for the preprint version only and contains no original work.

Let $a_1,\ldots,a_N$ be complex numbers,
$\lambda_1<\lambda_2<\cdots<\lambda_N$ be real numbers 
and 
$$
\Phi(t)=\sum_{k=1}^N a_ke^{2i\pi \lambda_k t}.
$$

Let $\eps>0$. By a lemma of Dirichlet (\cite[p 235]{Zy}, \cite{Gr}), there 
is an increasing sequence of integers $(M_n)_{n\geq 1}$ and,
for each $n\geq 1$ a finite family of integers 
$(N_{k,n})_{k=1,\ldots,N}$ such that
$$
\abs{\lambda_k-\frac{N_{k,n}}{M_n}}<\frac{\eps}{M_n}\qquad\mbox{for }k=1,\ldots,N
$$
which implies that
$$
\abs{e^{2i\pi \lambda_k t}-e^{2i\pi \frac{N_{k,n}}{M_n}t}}
\leq 2\pi\abs{\lambda_j-\frac{N_{k,n}}{M_n}}|t|\leq 2\pi\frac{\eps}{M_n}|t|\qquad\mbox{for }k=1,\ldots,N.
$$
Define the $M_n$-periodic function
$$
\Psi_n(t)=\sum_{k=1}^N a_ke^{2i\pi N_{k,n} t/M_n}
$$
and note that, for $t\in[-M_n/2,M_n/2]$,
$$
\abs{\Phi(t)-\Psi_n(t)}\leq \sum_{k=1}^N |a_k|\abs{e^{2i\pi \lambda_k t}-e^{2i\pi \frac{N_{k,n}}{M_n}t}}
\leq 2\pi \eps \sum_{k=1}^N |a_k|.
$$
As $\Psi_n$ is $M_n$-periodic, we may apply \ref{th:MPS} to obtain
\begin{eqnarray*}
\sum_{k=1}^N\frac{|a_k|}{k}&\leq&C_{MPS}\frac{1}{M_n}\int_{-M_n/2}^{M_n/2}|\Psi_n(t)|\,\mbox{d}t\\
&\leq& C_{MPS}\frac{1}{M_n}\int_{-M_n/2}^{M_n/2}|\Phi(t)|\,\mbox{d}t+
2\pi\eps\sum_{k=1}^N |a_k|.
\end{eqnarray*}
Letting $n\to+\infty$ and then $\eps\to 0$ we obtain Theorem \ref{th:HL}:
$$
\lim_{T\to+\infty}\frac{1}{T}\int_{-T/2}^{T/2}\left|\sum_{k=1}^N a_ke^{2i\pi \lambda_k t}\right|\,\mathrm{d}t
\geq C_{MPS}\sum_{k=1}^N\frac{|a_k|}{k}.
$$
Note that $C_{MPS}$ is the same constant as in Theorem \ref{th:MPS}.

\section{Igham's Theorem}
This appendix is for the preprint version only and contains no original work.

\begin{theorem}[Ingham]
Let $a_1,\ldots,a_N$ be complex numbers and $\lambda_1<\cdots,\lambda_N$ be real numbers with
$|\lambda_k-\lambda_\ell|\geq 1$ if $k\not=\ell$. For every $T>1$, 
\begin{equation}
\label{eq:ingham}
\frac{1}{T}\int_{-T/2}^{T/2}\abs{\sum_{k=1}^N a_ke^{2i\pi\lambda_kt}}^2\,\mbox{d}t\geq C(T)\sum_{k=1}^N|a_k|^2
\end{equation}
with
$$
C(T)=\begin{cases}
\dfrac{\pi^2}{8}\dfrac{T^2-1}{T^3}&\mbox{for }1<T\leq 2\\
\dfrac{3\pi^2}{64}&\mbox{for }T\geq 2\\
\end{cases}.
$$
\end{theorem}

\begin{proof}
We will do so in three steps. We first establish this inequality for $1<T\leq 2$.

Let $H$ be defined by $H(t)=\mathbf{1}_{[-1/2,1/2]}\cos \pi t$ and notice that
$$
\widehat{H}(\xi)=\frac{2}{\pi}\frac{\cos\pi\xi}{1-4\xi^2}.
$$
Note that $H$ being non-negative, $\widehat{H}$ is maximal at $0$ (which may be checked directly).

Next consider $H*H$ and notice that, as $H$ is non-negative, even, continous with support $[-1/2,1/2]$,
then $H*H$ is non-negative, even, continous with support $[-1,1]$ and its Fourier transform is
$$
\widehat{H*H}(\xi)=\widehat{H}^2.
$$
Next $H\in H^1(\R)$ with $H'=-\pi\mathbf{1}_{[-1/2,1/2]}\sin \pi t$ and
$$
\widehat{H'}(\xi)=4i\xi\frac{\cos\pi\xi}{1-4\xi^2}
$$
thus
$$
\widehat{H'*H'}(\xi)=-(2\pi\xi)^2\widehat{H}^2
$$
We now consider
$
G_T=\pi^2 T^2 H*H+H'*H'
$
so that $G_T$ is continuous, real valued, even and supported in $[-1,1]$.
$$
\widehat{G_T}=\pi^2\left(T^2-4\xi^2\right)\widehat{H}^2
$$
is even (so $G_T$ is the Fourier transform of $\widehat{G_T}$)
and in $L^1$. 
Further $\widehat{G}$ is non-negative on $[-T/2,T/2]$ and negative on $\R\setminus[-T/2,T/2]$.

This implies that
\begin{eqnarray*}
\int_{-T/2}^{T/2}\widehat{G_T}(t)\abs{\sum_{k=1}^Na_je^{2i\pi\lambda_k t}}^2\,\mbox{d}t
&\geq&\int_{\R}\widehat{G_T}(t)\abs{\sum_{k=1}^Na_je^{2i\pi\lambda_k t}}^2\,\mbox{d}t\\
&=&\sum_{k,\ell=1}^Na_k\overline{a_\ell}\int_{\R}\widehat{G_T}(t)e^{2i\pi(\lambda_k-\lambda_\ell)t}\,\mbox{d}t\\
&=&\sum_{k,\ell=1}^Na_k\overline{a_\ell}G_T(\lambda_k-\lambda_\ell)
=\sum_{k=1}^N|a_k|^2G_T(0).
\end{eqnarray*}
In the last line, we use that $|\lambda_k-\lambda_\ell|\geq 1$ when $k\not=\ell$ thus 
$G_T(\lambda_k-\lambda_\ell)=0$.

Now, for $\xi\in[-T/2,T/2]$,
$$
\widehat{G_T}(\xi)=\pi^2(T^2-4\xi^2)\widehat{H}^2(\xi)
\leq \pi^2(T^2-4\xi^2)\widehat{H}^2(0)\leq 4T^2
$$
while
$$
G_T(0)=\pi^2\int_{-1/2}^{1/2}T^2\cos^2\pi t-\sin^2\pi t\,\mbox{d}t
=\frac{\pi^2}{2}(T^2-1)
$$
which leads to
\begin{equation}
\int_{-T/2}^{T/2}\abs{\sum_{k=1}^Na_je^{2i\pi\lambda_k t}}^2\,\mbox{d}t
\geq\frac{\pi^2}{8}\frac{T^2-1}{T^2} \sum_{k=1}^N|a_k|^2.
\label{eq:ingham1}
\end{equation}

\smallskip

For $2\leq T\leq 6$, we simply write
$$
\int_{-T/2}^{T/2}\abs{\sum_{k=1}^Na_je^{2i\pi\lambda_k t}}^2\,\mbox{d}t
\geq \int_{-1}^{1}\abs{\sum_{k=1}^Na_je^{2i\pi\lambda_k t}}^2\,\mbox{d}t
\geq\frac{3\pi^2}{32}\sum_{k=1}^N|a_k|^2
$$
where the second inequality is \eqref{eq:ingham1} with $T=2$, establishing
\eqref{eq:ingham} with $C=\dfrac{3\pi^2}{32T}\geq\dfrac{\pi^2}{64}$.

\smallskip

Now let $T\geq 6$ and $M_T=[T/2]$ so that $M_T\geq \dfrac{T}{2}-1\geq \dfrac{T}{3}$. For $j=0,\ldots,M_T-1$, let 
$t_j=-T/2+j+1$ so that the intervals
$[t_j-1,t_j+1[$ are disjoint and $\dst\bigcup_{j=0}^{N-1}[t_j-1,t_j+1[\subset [-T/2,T/2]$ thus
\begin{eqnarray*}
\int_{-T/2}^{T/2}\abs{\sum_{k=1}^Nb_ke^{2i\pi\lambda_k t}}^2\,\mbox{d}t
&\geq&\sum_{j=0}^{M_T-1}\int_{t_j-1}^{t_j+1}\abs{\sum_{k=1}^Nb_ke^{2i\pi\lambda_k t}}^2\,\mbox{d}t\\
&=&\sum_{j=0}^{M_T-1}\int_{-1}^{1}\abs{\sum_{k=1}^Nb_ke^{2i\pi\lambda_k t_j}e^{2i\pi\lambda_k t}}^2\,\mbox{d}t.
\end{eqnarray*}
Now, apply \eqref{eq:ingham1} with $a_k=b_ke^{2i\pi\lambda_k t_j}$ and $T=2$ to get
$$
\frac{1}{T}\int_{-T/2}^{T/2}\abs{\sum_{j=1}^Nb_ke^{2i\pi\lambda_k t}}^2\,\mbox{d}t\geq 
\frac{3\pi^2}{32}\dfrac{M_T}{T}\sum_{k=1}^N|a_k|^2\geq \frac{\pi^2}{32}\sum_{k=1}^N|a_k|^2,
$$
establishing \eqref{eq:ingham} with $C=\dfrac{\pi^2}{32}$.
\end{proof}

Note that, as $\dfrac{M_T}{T}\to\dfrac{1}{2}$ we obtain a slightly better bound for the Besikovitch
norm
$$
\lim_{T\to+\infty}\frac{1}{T}\int_{-T/2}^{T/2}\abs{\sum_{j=1}^Nb_ke^{2i\pi\lambda_k t}}^2\,\mbox{d}t\geq 
\frac{3\pi^2}{64}\sum_{k=1}^N|a_k|^2.
$$

\smallskip

Ingham also prove an $L^1$-inequality that is weaker than Nazarov's result:

\begin{proposition}[Ingham]
Let $\gamma>0$ and $T>\dfrac{1}{\gamma}$.
Let $(\lambda_n)_{n\in\Z}$ be a sequence such that $\lambda_{n+1}-\lambda_n\geq\gamma$.
Then for every $N$ and every sequence $(a_n)_{n=-N,\ldots,N}$,
$$
\frac{1}{T}\int_{-T/2}^{T/2}\abs{\sum_{n=-N}^Na_ne^{2i\pi\lambda_nt}}\,\mathrm{d}t
\geq \frac{2}{\pi}\frac{T^2-1}{T^2}\max_{-N,\ldots,N}|a_n|.
$$
Further, if $(a_n)\in\ell^2(\Z)$, then
$$
\frac{1}{T}\int_{-T/2}^{T/2}\abs{\sum_{n\in\Z}a_ne^{2i\pi\lambda_nt}}\,\mathrm{d}t
\geq \frac{2}{\pi}\frac{T^2-1}{T^2}\|(a_n)_{n\in\Z}\|_\infty.
$$
\end{proposition}

\begin{proof}
The second inequality follows immediately from the first one.

Multiplying the sum in the integral by $e^{-2i\pi\lambda_0t}$,
we may replace the sequence $(\lambda_n)_n$ by $(\lambda_n-\lambda_0)_n$
and thus assume that $\lambda_0=0$.

We again consider $H(t)=\mathbf{1}_{[-1/2,1/2]}\cos\pi t$ and set
$$
H_T(t)=\dfrac{1}{T}H\left(\dfrac{t}{T}\right)=
\dfrac{1}{T}\mathbf{1}_{[-T/2,T/2]}\cos \dfrac{\pi}{T}t
$$
so that 
$$
\widehat{H_T}(\xi)=\frac{2}{\pi}\frac{\cos\pi T\xi}{1-4T^2\xi^2}.
$$

Let $\ell$ be such that $|a_\ell|=\max_{-N\leq n\leq N}|a_n|$.

Our first observation is that $0\leq H_T\leq\dfrac{1}{T}$ and is supported in $[-T/2,T/2]$ so that
\begin{eqnarray}
\frac{1}{T}\int_{-T/2}^{T/2}\abs{\sum_{n=-N}^Na_ne^{2i\pi\lambda_nt}}\,\mbox{d}t
&\geq& \int_{-T/2}^{T/2}\abs{\sum_{n=-N}^Na_ne^{2i\pi(\lambda_n-\lambda_\ell)t}H_T(t)}\,\mbox{d}t\nonumber\\
&\geq&\abs{\sum_{n=-N}^Na_n \int_{-T/2}^{T/2}H_T(t)e^{2i\pi(\lambda_n-\lambda_\ell)t}\,\mbox{d}t}\nonumber\\
&=&\abs{\sum_{n=-N}^Na_n\widehat{H_T}(\lambda_n-\lambda_\ell)}\nonumber\\
&\geq&|a_\ell||\widehat{H_T}(0)|
-\sum_{n\in\{-N,\ldots,N\}\setminus\{\ell\}}|a_n||\widehat{H_T}(\lambda_n-\lambda_\ell)|\nonumber\\
&\geq&\frac{2}{\pi}|a_\ell|\left(1-\frac{\pi}{2}\sum_{n\in\{-N,\ldots,N\}\setminus\{\ell\}}\frac{\pi}{2}|\widehat{H_T}(\lambda_n-\lambda_\ell)|\right).\label{inghamlinfty}
\end{eqnarray}
But, $\lambda_{n+1}-\lambda_n\geq 1$ and $T>1$ so that $4T^2|\lambda_n-\lambda_\ell|^2-1\geq T^2(4 |n-\ell|^2-1)$
\begin{eqnarray*}
\sum_{n\in\{-N,\ldots,N\}\setminus\{\ell\}}\frac{\pi}{2}|\widehat{H_T}(\lambda_n-\lambda_\ell)|
&\leq&\sum_{n\in\{-N,\ldots,N\}\setminus\{\ell\}}\frac{1}{4T^2|\lambda_n-\lambda_\ell|^2-1}\\
&\leq&\frac{2}{T^2}\sum_{m=1}^{+\infty}\frac{1}{4m^2-1}
=\frac{1}{T^2}\sum_{m=1}^{+\infty}\left(\frac{1}{2m-1}-\frac{1}{2m+1}\right)=\frac{1}{T^2}.
\end{eqnarray*}
Inseting this into \eqref{inghamlinfty} we obtain the proposition.
\end{proof}


\begin{thebibliography}{11}

\bibitem{BKL}
{\sc C. Baiocchi, Claudio; V. Komornik \& P. Loreti}
{\em Ingham type theorems and applications to control theory.}
Boll. Unione Mat. Ital. Sez. B Artic. Ric. Mat. (8) {\bf 2} (1999), 33--63.

\bibitem{BeKo}
{\sc A.\,S. Belov \& S.\,V. Konyagin}
{\em On the conjecture of Littlewood and minima of even trigonometric polynomials.}
Harmonic analysis from the Pichorides viewpoint (Anogia, 1995), 1–11, 
Publ. Math. Orsay, {\bf 96-01}, Univ. Paris XI, Orsay, 1996. 

\bibitem{CQ}
{\sc D. Choimet \& H. Queff\'elec}
Twelve landmarks of twentieth-century analysis.
Cambridge University Press, New York, 2015

\bibitem{Coh}
{\sc P.\,J. Cohen}
{\em On a conjecture of Littlewood and idempotent measures.}
Amer. J. Math. {\bf 82} (1960), 191-–212. 

\bibitem{Da}
{\sc H. Davenport}
{\em On a theorem of P. J. Cohen.}
Mathematika {\bf 7} (1960), 93--97. 

\bibitem{dvl}
{\sc R.\,A. DeVore \& G.\,G. Lorentz} 
Constructive Approximation.
Grundlehren der mathematischen Wissenschaften, 303. Springer-Verlag, Berlin, 1993.

\bibitem{Fo1}
{\sc J.\,J.\,F. Fournier}
{\em On a theorem of Paley and the Littlewood conjecture.}
Ark. Mat. {\bf 17} (1979), 199--216.
 
\bibitem{Fo2}
{\sc J.\,J.\,F. Fournier}
{\em Some remarks on the recent proofs of the Littlewood conjecture.}
Second Edmonton conference on approximation theory (Edmonton, Alta., 1982), 157–170, CMS Conf. Proc., 3, Amer. Math. Soc., Providence, RI, 1983.

\bibitem{1}
{\sc G.\,H. Hardy \& J.\,E. Littlewood}
{\em A new proof of a theorem on rearrangements.}
J. London Math. Soc. {\bf 23} (1948), 163–-168.

\bibitem{HL}
{\sc S. Hudson \& M. Leckband}
{\em Hardy's inequality and fractal measures.}
J. Funct. Anal. {\bf 108} (1992), 133--160.

\bibitem{MPS}
{\sc O.\,C. McGehee, L. Pigno \& B. Smith}
{\em Hardy's inequality and the $L^1$ norm of exponential sums.}
Ann. of Math. (2) {\bf 113} (1981), 613--618. 

\bibitem{Gr} 
{\sc L. Grafakos}
Classical Fourier analysis. Third edition. 
Graduate Texts in Mathematics, {\bf 249}. Springer, New York, 2014
 
\bibitem{In} 
{\sc A.\,E. Ingham}
{\em Some trigonometrical inequalities with applications to the theory of series.}
Math. Z. {\bf 41} (1936), 367--379.

\bibitem{Kon}
{\sc S.\,V. Konyagin}
{\em On the Littlewood problem.} (Russian)
Izv. Akad. Nauk SSSR Ser. Mat. {\bf 45} (1981), no. 2, 243–-265, 463. 

\bibitem{Naz}
{\sc F.\,L. Nazarov}
{\em On a proof of the Littlewood conjecture by McGehee, Pigno and Smith. }
St. Petersburg Math. J. {\bf 7} (1996), no. 2, 265-275 

\bibitem{Pi1}
{\sc S.\,K. Pichorides}
{\em A lower bound for the $L^1$ norm of exponential sums.}
Mathematika {\bf 21} (1974), 155–-159.

\bibitem{Pi2}
{\sc S.\,K. Pichorides}
{\em On a conjecture of Littlewood concerning exponential sums. I.}
Bull. Soc. Math. Grèce (N.S.) {\bf 18} (1977), 8--16.

\bibitem{Pi3}
{\sc S.\,K. Pichorides}
{\em On a conjecture of Littlewood concerning exponential sums. II.}
Bull. Soc. Math. Grèce (N.S.) {\bf 19} (1978),  274--277.

\bibitem{Pi4}
{\sc S.\,K. Pichorides}
{\em On the L1 norm of exponential sums.}
Ann. Inst. Fourier (Grenoble) {\bf 30} (1980),  79--89.

\bibitem{thesechadi}
{\sc C. Saba}
{\em Sur la norme $L^1$ de sommes d'exponentielles}
Doctorat de l'Universit\'e de Bordeaux, en pr\'eparation.

\bibitem{Steg}
{\sc J.\,D. Stegeman}
{\em On the constant in the Littlewood problem.}
Math. Ann. {\bf 261} (1982), 51–-54. 

\bibitem{TB}
{\sc R.\,M. Trigub \& E.\,S. Belinsky}
{\em Fourier Analysis and Appoximation of Functions}, Kluwer, 2004.

\bibitem{Yab}
{\sc K. Yabuta}
{\em A remark on the Littlewood conjecture.} 
Bull. Fac. Sci. Ibaraki Univ. Ser. A {\bf 14} (1982), 19--21

\bibitem{Zy}
{\sc A. Zygmund}
{\em Trigonometric Series, vol1.}
Cambridge 1959.
\end{thebibliography}
\end{document}